\pdfoutput=1
\documentclass[11pt]{article}

\usepackage[utf8]{inputenc}
\usepackage[T1]{fontenc}
\usepackage[a4paper,margin=1in]{geometry}
\usepackage{amsmath,amssymb,amsthm,mathtools,mathrsfs}
\usepackage{microtype}
\usepackage[colorlinks=true,linkcolor=blue,citecolor=blue,urlcolor=blue]{hyperref}
\hypersetup{
  pdftitle={High-index asymptotics of Mathieu spectral gaps and splittings},
  pdfauthor={Berkay Anahtarci},
  pdfkeywords={Mathieu operator; Hill operator; spectral gap; non-self-adjoint operator; high-index asymptotics; Lyapunov--Schmidt reduction; Dyck path}
}

\newtheorem{theorem}{Theorem}[section]
\newtheorem{proposition}[theorem]{Proposition}
\newtheorem{lemma}[theorem]{Lemma}
\newtheorem{corollary}[theorem]{Corollary}
\theoremstyle{definition}
\newtheorem{definition}[theorem]{Definition}
\theoremstyle{remark}
\newtheorem{remark}[theorem]{Remark}

\newcommand{\C}{\mathbb{C}}

\title{High-index asymptotics of Mathieu spectral gaps and splittings}
\author{Berkay Anahtarci\\[2pt]
\small Department of Mathematical Engineering, \"{O}zye\u{g}in University,\\
\small 34794 \c{C}ekmek\"{o}y, \.{I}stanbul, T\"{u}rkiye\\
\small Corresponding author: \texttt{berkay.anahtarci@ozyegin.edu.tr}}
\date{}

\begin{document}

\maketitle

\begin{center}
\emph{Dedicated to the memory of Plamen B. Djakov.}
\end{center}

\begin{abstract}
Consider the Mathieu operator $L=-d^2/dx^2+2a\cos(2x)$ for fixed
$a\in\C\setminus\{0\}$. For large enough $n$, the periodic ($n$ even)
or anti-periodic ($n$ odd) problem on $[0,\pi]$ has exactly two algebraically simple
eigenvalues $\lambda_n^-,\lambda_n^+$ in the disc $|\lambda-n^2|<1$. We prove
that, as $n\to\infty$,
\[
\lambda_n^+-\lambda_n^-
=\pm\frac{8(a/4)^n}{[(n-1)!]^2}
\exp\!\left[-\frac{a^2(n+2)}{4(n^2-1)^2}-\frac{15a^4}{128n^7}+O(n^{-8})\right],
\]
where the sign depends on the labelling. For real $a$ with
$\lambda_n^-<\lambda_n^+$, this difference is the length of the $n$th spectral
gap. Expanding the exponential determines
the relative corrections of orders $n^{-4}$ to $n^{-7}$, beyond the correction
$-a^2/(4n^3)$ obtained by Anahtarci and Djakov (2012) with a relative
$O(n^{-4})$ remainder. By isospectrality,
the same formula holds for the potential $ce^{-2ix}+de^{2ix}$ with
$a^2=cd\ne0$.
\end{abstract}

\vspace{0.8em}

\vspace{0.3em}

\section{Introduction}

Let $v\in L^2_{\mathrm{loc}}(\mathbb R)$ be a $\pi$-periodic potential with
mean zero. If $v$ is real, the Hill operator $-d^2/dx^2+v$ on $\mathbb R$ is
self-adjoint and its spectrum has a band--gap structure: there are points
$\lambda_0^+<\lambda_1^-\le\lambda_1^+<\lambda_2^-\le\lambda_2^+<\cdots$ such
that the spectrum is $\bigcup_{n\ge1}[\lambda_{n-1}^+,\lambda_n^-]$, and the
bands are separated by the spectral gaps $(\lambda_n^-,\lambda_n^+)$
\cite{MagnusWinkler1966,Eastham1974}. The gap edges $\lambda_n^\pm$ are the
eigenvalues of $-y''+vy=\lambda y$ on $[0,\pi]$ with periodic boundary
conditions for even $n$ and anti-periodic ones for odd $n$, and
$\gamma_n:=\lambda_n^+-\lambda_n^-$ is the length of the $n$th gap. The decay
of these lengths reflects the regularity of $v$: decay faster than every power
of $n^{-1}$ characterizes $C^\infty$, and exponential decay characterizes
analyticity \cite{Hochstadt1963,MarchenkoOstrovskii1975,Trubowitz1977}.
Grigis \cite{Grigis1987} determined the generic leading behaviour of
$\gamma_n$ for trigonometric-polynomial potentials, of which the Mathieu
potential is the simplest. If $v$ is not real-valued, the operator is not
self-adjoint and this band--gap description need not hold; still, for large
$n$, the periodic or anti-periodic problem has exactly two eigenvalues
$\lambda_n^-,\lambda_n^+$ in the disc $|\lambda-n^2|<1$, counted with algebraic multiplicity
\cite[Lemma~1]{AnahtarciDjakov2012}, and we keep the notation
$\gamma_n=\lambda_n^+-\lambda_n^-$.

We consider the Mathieu operator
\[
L=-\frac{d^2}{dx^2}+2a\cos(2x),\qquad a\in\C\setminus\{0\}\ \text{fixed},
\]
with the boundary conditions above, as $n\to\infty$. For real $a$, every gap
is open by Ince's theorem \cite{Ince1922}. For non-real $a$, we call
$\gamma_n$ the \emph{eigenvalue splitting}; exchanging
the labels changes its sign. Multiplication by $2a\cos(2x)$ shifts Fourier
indices by $\pm2$, so the resonant modes $e^{\pm inx}$ interact only through
chains of at least $n$ couplings. A Lyapunov--Schmidt reduction to their span
\cite{DjakovMityagin2006} turns the eigenvalue problem near $n^2$ into a
$2\times2$ problem whose diagonal entry is a sum over closed lattice walks
based at $n$ and whose off-diagonal entry is a sum over walks from $-n$ to $n$
(Section~\ref{sec:LS}). By Lemma~\ref{lem:center},
Proposition~\ref{prop:root} and \S\ref{subsec:completion}, both eigenvalues are displaced from $n^2$ by
$a^2/(2n^2)+a^2/(2n^4)+(16a^2+5a^4)/(32n^6)+O(n^{-8})$, refining
\cite[(3.24)]{AnahtarciDjakov2012}, whereas $\gamma_n$ is factorially small.
These common terms therefore do not determine the splitting, whose size is set
by the off-diagonal entry.

For real $a$, Harrell \cite{Harrell1981} obtained the size of $\gamma_n$ as
$n\to\infty$ up to a multiplicative constant, and Avron and Simon
\cite{AvronSimon1981} computed the constant, proving
$\gamma_n=8(|a|/4)^n[(n-1)!]^{-2}(1+o(n^{-2}))$; Hochstadt
\cite{Hochstadt1984} gave another proof of the leading asymptotics. Anahtarci and Djakov
\cite{AnahtarciDjakov2012} found the first relative correction for every
$a\in\C\setminus\{0\}$:
\begin{equation}\label{eq:ad}
\gamma_n
=\pm\frac{8(a/4)^n}{[(n-1)!]^2}
\left[1-\frac{a^2}{4n^3}+O(n^{-4})\right].
\end{equation}
For fixed $n$ and $a\to0$, Levy and Keller \cite{LevyKeller1963} found the
same leading factor, and the DLMF \cite{DLMF28.6} records small-coupling
expansions of the characteristic values. Ba\c{s}ar and Dunne
\cite{BasarDunne2015} organize the band edges and their splitting at fixed
level as series in powers of the coupling and recover the leading behaviour of
the gap at large level; rigorous large-coupling asymptotics are in
\cite{OgilvieOldeDaalhuis2015}. Our main result, Theorem~\ref{thm:main},
refines \eqref{eq:ad} for fixed $a$ as $n\to\infty$ by determining four
further relative corrections, of orders $n^{-4}$ through $n^{-7}$, with a
proved relative remainder $O(n^{-8})$.

\begin{theorem}[Spectral gap and splitting asymptotics]\label{thm:main}
Fix $a\in\C\setminus\{0\}$. There is an integer $n_0=n_0(a)$ such that, for
$n\ge n_0$, the operator $L$, with periodic boundary conditions for even $n$
and anti-periodic ones for odd $n$, has exactly two eigenvalues
$\lambda_n^-,\lambda_n^+$ in the disc $|\lambda-n^2|<1$; both are
algebraically simple, and
\begin{equation}\label{eq:main-refined}
\gamma_n=\lambda_n^+-\lambda_n^-
=\pm\frac{8(a/4)^n}{[(n-1)!]^2}
\exp\!\left[-\frac{a^2(n+2)}{4(n^2-1)^2}-\frac{15a^4}{128n^7}+O(n^{-8})\right].
\end{equation}
The implied constant depends only on $a$ and, for each $n$, the sign depends
on the labelling of the pair. If $a$ is real and $\lambda_n^-<\lambda_n^+$,
then \eqref{eq:main-refined} holds with the sign $+$ and with $(|a|/4)^n$ in
place of $(a/4)^n$; $\gamma_n$ is then the length of the $n$th spectral gap.
\end{theorem}

Expanding the exponential in \eqref{eq:main-refined} gives
\begin{equation}\label{eq:main-expanded}
\gamma_n=\pm\frac{8(a/4)^n}{[(n-1)!]^2}
\left[
1-\frac{a^2}{4n^3}-\frac{a^2}{2n^4}-\frac{a^2}{2n^5}
+\Bigl(\frac{a^4}{32}-a^2\Bigr)\frac{1}{n^6}
+\Bigl(\frac{a^4}{128}-\frac{3a^2}{4}\Bigr)\frac{1}{n^7}+O(n^{-8})
\right].
\end{equation}
The correction of order $n^{-3}$ agrees with \eqref{eq:ad}. For a real cosine
potential, Li \cite{Li2023} conjectured from high-accuracy numerical
computations an expansion of the logarithm of the gap length; after rescaling,
its terms through order $n^{-7}$ are those of \eqref{eq:main-refined}.

We use the Lyapunov--Schmidt and Fourier-walk method of
\cite{DjakovMityagin2006,DjakovMityagin2007} as in
\cite{AnahtarciDjakov2012}, following that paper for the reduction, the walk
series, the two shortest classes of closed walks, the walks from $-n$ to $n$
with at most one backstep and the isospectral transfer. Two of its error terms
leave contributions undetermined at the first order beyond \eqref{eq:ad}:
the walks with two or more backsteps are bounded together, relative to the
monotone walk, and the difference of the diagonal entry at the two eigenvalues
is computed only to its leading term
\cite[(4.25) and Lemma~7]{AnahtarciDjakov2012}. We determine both contributions
through relative order $n^{-7}$.

Section~\ref{sec:LS} recalls the reduction. Section~\ref{sec:diagonal}
expands the diagonal entry by pairing the two closed walks that realize each
fixed primitive Dyck excursion, one on each side of $n$, and locates the
spectral centre, the eigenvalue shift that this entry alone would produce,
together with the inverse slope there. Section~\ref{sec:offdiag} proves that
the off-diagonal entry equals the contribution of the monotone walk divided by
a normalized tridiagonal determinant. Section~\ref{sec:proof} expands both
eigenvalues about the common centre. Up to sign and a factorially small
relative error, the splitting is then twice the off-diagonal entry at the
centre times the inverse slope, and combining the results of
Sections~\ref{sec:diagonal} and~\ref{sec:offdiag} yields
\eqref{eq:main-refined}. Corollary~\ref{cor:asymmetric} transfers
Theorem~\ref{thm:main} to $ce^{-2ix}+de^{2ix}$ with $a^2=cd\ne0$, by Veliev's
isospectrality theorem \cite{Veliev2013Isospectral}. Appendices~\ref{app:sums}
and~\ref{app:coupling} evaluate the coefficients and give fixed-index checks.

\section{Lyapunov--Schmidt reduction}\label{sec:LS}

We recall the Lyapunov--Schmidt reduction of
\cite[Lemma~21 and \S2.2]{DjakovMityagin2006}, specialized to the Mathieu
potential in \cite[\S2]{AnahtarciDjakov2012}.

Fix $n\ge2$. The operator $L=L_0+V$, where $L_0=-d^2/dx^2$ and $V$ is
multiplication by $2a\cos(2x)$, acts in $L^2(0,\pi)$ on the functions
$y\in H^2(0,\pi)$ with $y^{(j)}(\pi)=(-1)^ny^{(j)}(0)$, $j=0,1$: periodic
boundary conditions for even $n$, anti-periodic ones for odd $n$. The functions $e^{ikx}$,
$k\in n+2\mathbb Z$, form an orthogonal basis of $L^2(0,\pi)$ satisfying
these conditions, and
\[
L_0e^{ikx}=k^2e^{ikx},\qquad
Ve^{ikx}=a\bigl(e^{i(k+2)x}+e^{i(k-2)x}\bigr).
\]
Thus $n^2$ is a double eigenvalue of $L_0$, with eigenspace
$E_n=\operatorname{span}\{e^{inx},e^{-inx}\}$; we call $e^{\pm inx}$ the
\emph{resonant} modes. Let $P_n$ be the orthogonal
projection onto $E_n$ and $P_n^\perp=I-P_n$. Since $n\ge2$, $V$ does not
couple the resonant modes directly, so $P_nVP_n=0$.

Write $\lambda=n^2+z$ and fix $R\in(0,1)$. The reduction below assumes
\begin{equation}\label{eq:H}
|z|\le R,\qquad 2|a|<4(n-1)-R.
\end{equation}
For fixed $a$ this holds for all $n\ge n_0(a,R)$, and for fixed $n$ it holds
for small $|a|$, the regime of Proposition~\ref{prop:b2}.

On $\operatorname{Dom}(L)\cap\operatorname{Ran}P_n^\perp$, let $\Lambda_n(z)$ act by
$e^{ikx}\mapsto(n^2-k^2+z)e^{ikx}$, $k\ne\pm n$, and set
\[
D_n(z):=P_n^\perp(\lambda-L)P_n^\perp=\Lambda_n(z)-P_n^\perp VP_n^\perp.
\]
For $k=n+2m$ with $m\ne0,-n$, we have $|n^2-k^2|=4|m(n+m)|\ge4(n-1)$, so
$\|\Lambda_n(z)^{-1}\|\le(4(n-1)-R)^{-1}$. Since $\|V\|\le2|a|$,
\eqref{eq:H} gives $\|\Lambda_n(z)^{-1}P_n^\perp VP_n^\perp\|<1$. Hence
$D_n(z)$ maps $\operatorname{Dom}(L)\cap\operatorname{Ran}P_n^\perp$
bijectively onto $\operatorname{Ran}P_n^\perp$, and
\begin{equation}\label{eq:complementary-neumann}
D_n(z)^{-1}
=\sum_{m\ge0}\bigl[\Lambda_n(z)^{-1}P_n^\perp VP_n^\perp\bigr]^m
\Lambda_n(z)^{-1}.
\end{equation}

For $y\in\operatorname{Dom}(L)$, write $y=u+y_\perp$ with $u=P_ny$ and
$y_\perp=P_n^\perp y$. Since $P_nVP_n=0$ and $L_0$ is diagonal,
$(\lambda-L)y=0$ is equivalent to
\[
zu=P_nVy_\perp,\qquad D_n(z)y_\perp=P_n^\perp Vu.
\]
Eliminating $y_\perp=D_n(z)^{-1}P_n^\perp Vu$ gives $S_n(z)u=0$, where
$S_n(z)\colon E_n\to E_n$,
\[
S_n(z):=zI_{E_n}-P_nVD_n(z)^{-1}P_n^\perp VP_n,
\]
is the Schur complement of $D_n(z)$ in $\lambda-L$. Thus
$u\mapsto u+D_n(z)^{-1}P_n^\perp Vu$ is a bijection from $\ker S_n(z)$ onto
$\ker(\lambda-L)$, with inverse $y\mapsto P_ny$. The entries of $S_n(z)$ are
sums over the following walks.

\begin{definition}[Admissible Fourier walk and weight]\label{def:fourier-walk}
An \emph{admissible walk} of length $\ell\ge2$ is a sequence
$x=(j_0,\ldots,j_\ell)$ in $n+2\mathbb Z$ with $|j_s-j_{s-1}|=2$,
$j_0,j_\ell\in\{\pm n\}$, and $j_s\notin\{\pm n\}$ for $0<s<\ell$. It is
\emph{diagonal} if $j_\ell=j_0$ and \emph{off-diagonal} if $j_\ell=-j_0$.
Its weight is
\begin{equation}\label{eq:fourier-walk-weight}
w(x,z)=a^\ell\prod_{s=1}^{\ell-1}\frac{1}{n^2-j_s^2+z}.
\end{equation}
\end{definition}

In the Fourier basis, the $m$th term of \eqref{eq:complementary-neumann}
contributes to the entries of $P_nVD_n(z)^{-1}P_n^\perp VP_n$ exactly the
weights of the admissible walks of length $m+2$: each factor $V$ is a step,
each factor $\Lambda_n(z)^{-1}$ supplies the denominator at an interior
vertex, and $P_n^\perp$ excludes interior visits to $\pm n$. The weights
depend on the vertices only through $j_s^2$, so the reflection
$j_s\mapsto-j_s$ preserves them. Hence, in the basis $(e^{inx},e^{-inx})$,
\begin{equation}\label{eq:Sn-walks}
\begin{gathered}
S_n(z)=
\begin{pmatrix}
z-\alpha_n(z)&-\beta_n(z)\\
-\beta_n(z)&z-\alpha_n(z)
\end{pmatrix},\\[4pt]
\alpha_n(z)=\sum_{x:\,n\to n}w(x,z),
\qquad
\beta_n(z)=\sum_{x:\,-n\to n}w(x,z),
\end{gathered}
\end{equation}
where the sums run over all admissible walks from $n$ to $n$ and from $-n$ to
$n$, respectively. Each interior denominator in \eqref{eq:fourier-walk-weight}
has modulus at least $4(n-1)-|z|$, and at most $2^\ell$ walks of length $\ell$
start at a given point. Hence both series converge absolutely and locally
uniformly on the open set $2|a|+|z|<4(n-1)$, which contains the set defined by
\eqref{eq:H}, and $\alpha_n$, $\beta_n$ are analytic in $z$ there, in
particular on a neighbourhood of $|z|\le R$. Every off-diagonal walk has
length at least $n$; among the walks from $-n$ to $n$, equality holds only for
the \emph{monotone} walk, which has no negative step.

Define
\begin{equation}\label{eq:det}
\Delta_n(z):=\det S_n(z)
=\bigl(z-\alpha_n(z)\bigr)^2-\beta_n(z)^2.
\end{equation}
By the bijection above, under \eqref{eq:H},
\begin{equation}\label{eq:eig-equiv}
\Delta_n(z)=0
\quad\Longleftrightarrow\quad
n^2+z\ \text{is an eigenvalue of }L.
\end{equation}
The bijection identifies eigenvalues and eigenspaces; by itself it does not
identify algebraic multiplicities, and Section~\ref{sec:proof} obtains
simplicity from two distinct eigenvalues located through
\eqref{eq:eig-equiv}, together with an eigenvalue count.

\section{Diagonal estimates}\label{sec:diagonal}

Throughout Sections~\ref{sec:diagonal}--\ref{sec:proof}, we fix
$a\in\C\setminus\{0\}$; thresholds $n_0$ and constants $C$ depend only on $a$
and on the parameters fixed in each statement. All logarithms are principal.

By \eqref{eq:Sn-walks}, $\alpha_n(z)=\sum_{p\ge1}A_{2p-1}(n,z)$, where
$A_{2p-1}(n,z)$ is the sum of the weights of the diagonal walks of length $2p$
based at $n$. For $n>p$, we express $A_{2p-1}(n,z)$ as a sum over primitive Dyck
excursions of semilength $p$, each realized above and below $n$.
Pairing the two weights gives $A_{2p-1}(n,z)=O(n^{-2p})$ for fixed $p$.
Thus, to obtain an $O(n^{-8})$ remainder, we evaluate $A_1$, $A_3$
and $A_5$ and bound the sum of the remaining classes.

\subsection{Pairing of diagonal walks}

Since steps have size $2$ within $n+2\mathbb Z$ and interior visits to $\pm n$
are excluded, a diagonal walk based at $n$ stays strictly on one side of $n$:
it is \emph{outward} ($j_s>n$) or \emph{inward} ($j_s<n$) for $0<s<2p$. In
either case its heights $h_s=|j_s-n|/2$ form a primitive Dyck excursion
$\eta=(h_0,\ldots,h_{2p})$ of semilength $p$: $h_0=h_{2p}=0$,
$|h_s-h_{s-1}|=1$, and $h_s>0$ for $0<s<2p$ \cite[Chapter~6]{Stanley2023}.
Let $\mathcal E_p$ be the finite set of primitive Dyck excursions of
semilength $p$. Conversely, each $\eta\in\mathcal E_p$ is realized by the
outward walk $j_s=n+2h_s$ and, when $n>p$, by its reflection in $n$, the inward
walk $j_s=n-2h_s$, which stays above $-n$ because $h_s\le p<n$. Since $n^2-(n\mp2h)^2=4h(\pm n-h)$, the inward weight
is
\[
W_\eta(n,z):=a^{2p}\prod_{s=1}^{2p-1}\frac{1}{4h_s(n-h_s)+z},
\]
and the outward weight is $W_\eta(-n,z)$.

\begin{proposition}[Paired diagonal classes]\label{prop:fixed-diagonal-bounds}
Fix an integer $p\ge1$. For $n>p$,
\begin{equation}\label{eq:diagonal-pairing}
A_{2p-1}(n,z)=\sum_{\eta\in\mathcal E_p}\bigl[W_\eta(n,z)+W_\eta(-n,z)\bigr],
\end{equation}
and there is a constant $C_p$, depending only on $p$, such that for $n\ge2p$,
\begin{equation}\label{eq:fixed-diagonal-cancellation}
\sup_{|z|\le1}|A_{2p-1}(n,z)|\le C_p|a|^{2p}n^{-2p}.
\end{equation}
\end{proposition}

\begin{proof}
For $n>p$, every diagonal walk of length $2p$ based at $n$
corresponds uniquely to a side and an excursion in $\mathcal E_p$,
and every such pair gives an admissible walk. Summing the inward
weights $W_\eta(n,z)$ and the outward weights $W_\eta(-n,z)$ over
$\eta\in\mathcal E_p$ gives \eqref{eq:diagonal-pairing}.
For the bound, fix $\eta\in\mathcal E_p$, $n\ge2p$ and $|z|\le1$, and put
$\Pi^+=\prod_{s=1}^{2p-1}\bigl(4h_s(n-h_s)+z\bigr)$ and
$\Pi^-=\prod_{s=1}^{2p-1}\bigl(4h_s(n+h_s)-z\bigr)$. Then
$W_\eta(n,z)=a^{2p}/\Pi^+$ and, since $2p-1$ is odd,
$W_\eta(-n,z)=-a^{2p}/\Pi^-$, so
\[
W_\eta(n,z)+W_\eta(-n,z)=a^{2p}\,\frac{\Pi^--\Pi^+}{\Pi^+\Pi^-}.
\]
As $1\le h_s\le p\le n/2$, each factor of $\Pi^\pm$ has modulus between $n$
and $7pn$, and corresponding factors differ by $8h_s^2-2z$, of modulus at most
$8p^2+2$. Writing $\Pi^--\Pi^+$ as a telescoping sum of $2p-1$ terms, each
containing one such difference, gives
$|\Pi^--\Pi^+|\le(2p-1)(8p^2+2)(7pn)^{2p-2}$, while
$|\Pi^+\Pi^-|\ge n^{4p-2}$. Hence
$|W_\eta(n,z)+W_\eta(-n,z)|\le(2p-1)(8p^2+2)(7p)^{2p-2}|a|^{2p}n^{-2p}$, and
summing over the finite set $\mathcal E_p$ gives
\eqref{eq:fixed-diagonal-cancellation}.
\end{proof}

Each $W_\eta(n,z)$ is of order $n^{1-2p}$, but the sum over the two
realizations keeps only its even part in $n$ and gains one power of $n$. For
$p=3$ this gives $A_5=O(n^{-6})$.

\subsection{The classes \texorpdfstring{$A_1$, $A_3$, $A_5$}{A1, A3, A5} and the tail}

To evaluate the first three classes, put
$e_j^\pm(z):=4j(\pm n-j)+z$, with $+$ for the inward realization
and $-$ for the outward one. The primitive Dyck excursions of
semilength $1$ and $2$ are $(0,1,0)$ and $(0,1,2,1,0)$; those of
semilength $3$ are $(0,1,2,3,2,1,0)$ and $(0,1,2,1,2,1,0)$.
Substituting them into \eqref{eq:diagonal-pairing} gives
$A_1=a^2f_n$ for $n\ge2$, $A_3=a^4g_n$ for $n\ge3$ and
$A_5=a^6h_n$ for $n\ge4$, where
\begin{equation}\label{eq:A1}
f_n(z)=\sum_{\pm}\frac{1}{e_1^\pm(z)}
=\frac{1}{4(n-1)+z}+\frac{1}{-4(n+1)+z}
=\frac{8-2z}{(4n)^2-(4-z)^2},
\end{equation}
\begin{equation}\label{eq:A3}
g_n(z)=\sum_{\pm}\frac{1}{(e_1^\pm)^2e_2^\pm},
\end{equation}
\begin{equation}\label{eq:A5}
h_n(z)=\sum_{\pm}\biggl[\frac{1}{(e_1^\pm)^2(e_2^\pm)^2e_3^\pm}
+\frac{1}{(e_1^\pm)^3(e_2^\pm)^2}\biggr].
\end{equation}
By Proposition~\ref{prop:fixed-diagonal-bounds} with $p=1,2,3$ and $a=1$,
$f_n=O(n^{-2})$, $g_n=O(n^{-4})$ and $h_n=O(n^{-6})$ uniformly on $|z|\le1$.
Hence, by Cauchy's estimate, for fixed $0<\rho<1$ and $k\ge0$,
\begin{equation}\label{eq:fg-bounds}
\sup_{|z|\le\rho}\bigl|f_n^{(k)}(z)\bigr|=O(n^{-2}),\qquad
\sup_{|z|\le\rho}\bigl|g_n^{(k)}(z)\bigr|=O(n^{-4}),\qquad
\sup_{|z|\le\rho}\bigl|h_n^{(k)}(z)\bigr|=O(n^{-6}).
\end{equation}
Put $\tau_n(z):=\sum_{p\ge4}A_{2p-1}(n,z)$.
For $n\ge4$ and $z$ in the convergence region of Section~\ref{sec:LS},
\[
\alpha_n(z)=a^2f_n(z)+a^4g_n(z)+a^6h_n(z)+\tau_n(z).
\]

\begin{proposition}[Diagonal tail]\label{prop:diagonal-tail-uniform}
There is a constant $C$ such that for $n\ge\max\{8,2|a|\}$,
\[
\sup_{|z|\le1}|\tau_n(z)|\le Cn^{-8}.
\]
Consequently, for each fixed $0<\rho<1$ and integer $k\ge0$,
\begin{equation}\label{eq:diagonal-derivative-bounds}
\sup_{|z|\le\rho}\bigl|\tau_n^{(k)}(z)\bigr|=O(n^{-8}),\qquad
\sup_{|z|\le\rho}\bigl|\alpha_n^{(k)}(z)\bigr|=O(n^{-2}).
\end{equation}
\end{proposition}

\begin{proof}
For $n\ge\max\{8,2|a|\}$ the closed unit disc lies in the region
$2|a|+|z|<4(n-1)$ of Section~\ref{sec:LS}, where the walk series are
analytic. For $p=4$, Proposition~\ref{prop:fixed-diagonal-bounds} gives
$O(n^{-8})$. For $p\ge5$, a walk of length $2p$ has at most $4^p$ step
sequences, and each interior denominator satisfies
$|n^2-j_s^2+z|\ge4(n-1)-1\ge2n$. Hence
$\sup_{|z|\le1}|A_{2p-1}(n,z)|\le4^p|a|^{2p}(2n)^{1-2p}=2n(|a|^2/n^2)^p$,
and the sum over $p\ge5$ is $O(n^{-9})$ since $|a|^2/n^2\le\frac14$. This
proves the bound on $\tau_n$. Cauchy's estimate gives the bound on
$\tau_n^{(k)}$ in \eqref{eq:diagonal-derivative-bounds}, and the decomposition
of $\alpha_n$ above, with \eqref{eq:fg-bounds}, gives the bound on
$\alpha_n^{(k)}$.
\end{proof}

Applied to every $p\ge4$, the count gives
$\sum_{p\ge4}\sup_{|z|\le1}|A_{2p-1}(n,z)|=O(n^{-7})$, recovering
\cite[(3.23)]{AnahtarciDjakov2012}; pairing the $p=4$ class gains one
further power of $n$.

\subsection{Spectral centre and inverse slope}

Put $F_n(z):=z-\alpha_n(z)$, so that
$\Delta_n(z)=F_n(z)^2-\beta_n(z)^2$ by \eqref{eq:det}.
We locate the zero $z_n^0$ of $F_n$ near $0$ and expand the inverse slope
$1/F_n'(z_n^0)$ used in Section~\ref{sec:proof}.

Matching successive powers of $a^2$ in the fixed-point equation
$z=\alpha_n(z)$ gives the coefficients $\xi_1$ for $n\ge2$, $\xi_2$ for
$n\ge3$ and $\xi_3$ for $n\ge4$:
\begin{equation}\label{eq:xi}
\begin{gathered}
\xi_1:=f_n(0)=\frac{1}{2(n^2-1)},\qquad
\xi_2:=f_n'(0)\,\xi_1+g_n(0)=\frac{5n^2+7}{32(n^2-4)(n^2-1)^3},\\
\xi_3:=f_n'(0)\,\xi_2+\tfrac12f_n''(0)\,\xi_1^2+g_n'(0)\,\xi_1+h_n(0)
=\frac{9n^4+58n^2+29}{64(n^2-9)(n^2-4)(n^2-1)^5}.
\end{gathered}
\end{equation}
The closed forms follow from \eqref{eq:A1}--\eqref{eq:A5}; in $\xi_2$, the terms $-1/(16n^4)$ of
$f_n'(0)\xi_1$ and $1/(16n^4)$ of $g_n(0)$ cancel.

\begin{lemma}[Spectral centre]\label{lem:center}
For large $n$, $F_n$ has a unique zero $z_n^0$ in
$|z|\le\frac12$, and
\begin{equation}\label{eq:center-trunc}
z_n^0=\xi_1a^2+\xi_2a^4+\xi_3a^6+O(n^{-8}).
\end{equation}
In particular,
\begin{equation}\label{eq:center-exp}
z_n^0=\frac{a^2}{2n^2}+\frac{a^2}{2n^4}+\frac{16a^2+5a^4}{32n^6}+O(n^{-8}).
\end{equation}
\end{lemma}

\begin{proof}
By Proposition~\ref{prop:diagonal-tail-uniform} with $\rho=\frac12$,
$\alpha_n$ and $\alpha_n'$ are $O(n^{-2})$ and $\tau_n$ is $O(n^{-8})$,
uniformly on $|z|\le\frac12$. Hence, for large $n$, $\alpha_n$ is
a $\frac12$-contraction of $\{|z|\le\frac12\}$ into itself, and its unique
fixed point $z_n^0$ is the unique zero of $F_n$ there.

Put $w:=\xi_1a^2+\xi_2a^4+\xi_3a^6$; by \eqref{eq:fg-bounds},
$\xi_r=O(n^{-2r})$, so $w=O(n^{-2})$. Expanding $f_n$, $g_n$ and $h_n$ about
$0$ to orders $2$, $1$ and $0$, with the remainders bounded by
\eqref{eq:fg-bounds}, and using $\tau_n(w)=O(n^{-8})$, gives
\[
\alpha_n(w)=a^2\bigl[\xi_1+f_n'(0)w+\tfrac12f_n''(0)w^2\bigr]
+a^4\bigl[g_n(0)+g_n'(0)w\bigr]+a^6h_n(0)+O(n^{-8}).
\]
With $w$ substituted, the terms of order $a^2$, $a^4$ and $a^6$ are those of $w$
by \eqref{eq:xi}, and the rest are $O(n^{-8})$; thus $\alpha_n(w)=w+O(n^{-8})$.
Since $z_n^0=\alpha_n(z_n^0)$ and $\alpha_n$ is a $\frac12$-contraction,
$|z_n^0-w|\le\frac12|z_n^0-w|+|\alpha_n(w)-w|$, which gives
\eqref{eq:center-trunc}. Finally, $\xi_1=1/(2n^2)+1/(2n^4)+1/(2n^6)+O(n^{-8})$,
$\xi_2=5/(32n^6)+O(n^{-8})$ and $\xi_3=O(n^{-10})$ turn
\eqref{eq:center-trunc} into \eqref{eq:center-exp}.
\end{proof}

We call $z_n^0$ the \emph{spectral centre}, the eigenvalue shift
obtained by omitting the off-diagonal entry in the reduction.

For $r=1,2,3$ and the corresponding range of $n$ in
\eqref{eq:xi}, let $\omega_r$ be the defining expression for $\xi_r$ with
$f_n,g_n,h_n$ replaced by $f_n',g_n',h_n'$ and the coefficients $\xi_j$,
$j<r$, kept unchanged. The coefficients of $a^2$, $a^4$ and $a^6$ in
$-\log(1-\omega_1a^2-\omega_2a^4-\omega_3a^6)$ are
\begin{equation}\label{eq:s-def}
s_1:=\omega_1,\qquad s_2:=\omega_2+\tfrac12\omega_1^2,\qquad
s_3:=\omega_3+\omega_1\omega_2+\tfrac13\omega_1^3.
\end{equation}

\begin{proposition}[Inverse slope at the centre]\label{prop:alpha-derivative}
For large $n$, $F_n'(z_n^0)\ne0$,
\begin{equation}\label{eq:inverse-alpha}
\log\frac{1}{F_n'(z_n^0)}=s_1a^2+s_2a^4+s_3a^6+O(n^{-8}),
\end{equation}
and $s_r=O(n^{-2r})$ for $r=1,2,3$.
\end{proposition}

\begin{proof}
Put $w:=\xi_1a^2+\xi_2a^4+\xi_3a^6=O(n^{-2})$. By \eqref{eq:center-trunc} and
\eqref{eq:diagonal-derivative-bounds}, $\alpha_n'(z_n^0)=\alpha_n'(w)+O(n^{-10})$.
Since $\alpha_n'=a^2f_n'+a^4g_n'+a^6h_n'+\tau_n'$, the Taylor expansion of
$\alpha_n(w)$ in the proof of Lemma~\ref{lem:center}, applied to
$f_n',g_n',h_n'$, which satisfy the same bounds \eqref{eq:fg-bounds}, and
with $\tau_n'=O(n^{-8})$ by \eqref{eq:diagonal-derivative-bounds}, gives, by
the definition of $\omega_r$,
$\alpha_n'(z_n^0)=\omega_1a^2+\omega_2a^4+\omega_3a^6+O(n^{-8})$ with
$\omega_r=O(n^{-2r})$. In particular $\alpha_n'(z_n^0)=O(n^{-2})$, so
$F_n'(z_n^0)=1-\alpha_n'(z_n^0)\ne0$. With the principal branch of the
logarithm, $\log(1/F_n'(z_n^0))=-\log(1-x)$ with $x=\alpha_n'(z_n^0)$, and
the series $-\log(1-x)=x+\frac12x^2+\frac13x^3+O(x^4)$ gives
\eqref{eq:inverse-alpha}, since the terms of degree at least eight in $a$ are
$O(n^{-8})$. By \eqref{eq:s-def}, the bounds on $\omega_r$ give
$s_r=O(n^{-2r})$.
\end{proof}

\section{Off-diagonal estimates}\label{sec:offdiag}

The monotone walk carries the factorial size of $\beta_n$, so we measure
all walks from $-n$ to $n$ relative to it. We first group them by their
backsteps, each of which contributes a local factor
(\S\ref{subsec:backsteps}). Relative to the monotone walk, the sum over all
classes, the \emph{backstep factor}, is the reciprocal of a tridiagonal determinant, whose logarithm is a
sum over closed walks on a path graph (\S\ref{subsec:logdet}). We then
expand the logarithms of the monotone-walk ratio and the backstep factor at
the centre $z_n^0$ of Lemma~\ref{lem:center}.

\subsection{Backstep classes}\label{subsec:backsteps}

For $p\ge0$, let $\mathcal X_n(p)$ be the set of off-diagonal walks from $-n$
to $n$ (Definition~\ref{def:fourier-walk}) with exactly $p$ negative steps,
or \emph{backsteps}. Such a walk has $n+p$ positive steps and length $n+2p$.
By \eqref{eq:Sn-walks},
\[
\beta_n(z)=\sum_{p\ge0}\sigma_p(n,z),
\qquad
\sigma_p(n,z)=\sum_{x\in\mathcal X_n(p)}w(x,z).
\]
In the height coordinate $u=(j+n)/2$, these walks run from $0$ to $n$ with
unit steps. A walk that went below $0$ or above $n$ would visit $0$ or $n$ at
an interior step, which is forbidden; so the interior heights lie in
$\{1,\ldots,n-1\}$. A vertex at
height $k$ contributes the denominator
\begin{equation}\label{eq:dk}
d_k(z):=n^2-(2k-n)^2+z=4k(n-k)+z.
\end{equation}
The class $\mathcal X_n(0)$ consists of the monotone walk alone. Since
$\prod_{k=1}^{n-1}d_k(0)=4^{n-1}[(n-1)!]^2$,
\begin{equation}\label{eq:sigma0}
\sigma_0(n,z)=\frac{a^n}{\prod_{k=1}^{n-1}d_k(z)},
\qquad
\sigma_0(n,0)=\frac{4(a/4)^n}{[(n-1)!]^2}.
\end{equation}
Set $M_n:=|\sigma_0(n,0)|$. Put $D_k:=d_k(0)=4k(n-k)$ and
$U_j:=\sum_{k=1}^{n-1}D_k^{-j}$. Since
$1/(k(n-k))=(1/k+1/(n-k))/n$, we have $U_1=H_{n-1}/(2n)$, where
$H_m:=\sum_{r=1}^mr^{-1}$. For $2\le k\le n-1$, put
\begin{equation}\label{eq:qk}
q_k(z)=\frac{1}{d_{k-1}(z)d_k(z)},
\qquad
Q_n(z)=\sum_{k=2}^{n-1}q_k(z).
\end{equation}
Since $|d_k(z)|\ge2k(n-k)$ for $|z|\le1$ and $q_{n+1-k}=q_k$, splitting
the sum at the midpoint gives
\begin{equation}\label{eq:Omega}
\sum_{k=2}^{n-1}\sup_{|z|\le1}|q_k(z)|=O(n^{-2}).
\end{equation}

\begin{lemma}[Backstep sources]\label{lem:backsteps}
A walk $x\in\mathcal X_n(p)$ is determined by the heights $k_1,\ldots,k_p$
from which its backsteps start, in order of occurrence. These satisfy
$2\le k_i\le n-1$ and $k_{i+1}\ge k_i-1$, every such sequence arises, and
\[
\frac{w(x,z)}{\sigma_0(n,z)}=\prod_{i=1}^p a^2q_{k_i}(z).
\]
\end{lemma}

\begin{proof}
Between backsteps the walk only ascends. After its backstep from $k_i$ it is
at height $k_i-1$, ascends to $k_{i+1}\ge k_i-1$, and steps down again, so the
sequence $(k_i)$ determines the walk. The first and last steps ascend, so
each backstep joins the two interior heights $k_i$ and $k_i-1$ in
$\{1,\ldots,n-1\}$; hence $2\le k_i\le n-1$, and conversely every such
sequence defines an admissible walk. Since the walk starts at $0$ and ends at $n$, each height
$h\in\{1,\ldots,n-1\}$ is visited $1+\varepsilon_h+\varepsilon_{h+1}$ times, where $\varepsilon_h$ is the
number of backsteps from $h$. Thus, relative to the monotone walk, each
backstep from $k$ adds two steps and one visit at each of $k-1$ and $k$,
that is, the factor $a^2/(d_{k-1}d_k)=a^2q_k$.
\end{proof}

By Lemma~\ref{lem:backsteps}, the one-backstep walks are indexed by
$2\le k\le n-1$, so $\sigma_1/\sigma_0=a^2Q_n$, as in
\cite{AnahtarciDjakov2012}. In general, $\sigma_p/\sigma_0$ is $a^{2p}$ times
the sum of $q_{k_1}\cdots q_{k_p}$ over the sequences $k_1,\ldots,k_p$ of the
lemma. The classes with $p\ge2$ are not evaluated separately: the next
proposition sums all classes at once.

\subsection{The backstep factor as a determinant}\label{subsec:logdet}

Let $\mathcal D(z)=\operatorname{diag}(d_1(z),\dots,d_{n-1}(z))$, with $d_k$ as
in \eqref{eq:dk}, and let $\mathcal A$ be the adjacency matrix of the path
graph on $\{1,\dots,n-1\}$; on its edge $\{k-1,k\}$, the two entries of
$\mathcal D(z)^{-1}\mathcal A$ have product $q_k(z)$. For
$r\ge1$, put
$T_r(z):=\frac1{2r}\operatorname{tr}
\bigl[(\mathcal D(z)^{-1}\mathcal A)^{2r}\bigr]$.
In $\operatorname{tr}[(\mathcal D^{-1}\mathcal A)^{2r}]$, a closed walk
$v_0\to v_1\to\cdots\to v_{2r}=v_0$ contributes $\prod_{i<2r}d_{v_i}^{-1}$.
Pairing each step $k-1\to k$ with a step $k\to k-1$ turns this product into a
product of factors $q_k$, one for each pair, over consecutive edges. On one
edge there are $2$ closed walks of each even length. For a fixed pair of
adjacent edges, the walks using both edges number $4$ at length four, each of
weight $q_{k-1}q_k$, and $12$ at length six, half of weight $q_{k-1}^2q_k$ and
half of weight $q_{k-1}q_k^2$. For three consecutive edges, $6$ walks of
length six use all three. Dividing by $2r$ gives
\begin{equation}\label{eq:T-local}
\begin{aligned}
T_1&=Q_n,\qquad
T_2=\frac12\sum_{k=2}^{n-1}q_k^2+\sum_{k=3}^{n-1}q_{k-1}q_k,\\
T_3&=\frac13\sum_{k=2}^{n-1}q_k^3+\sum_{k=3}^{n-1}q_{k-1}q_k(q_{k-1}+q_k)
+\sum_{k=4}^{n-1}q_{k-2}q_{k-1}q_k.
\end{aligned}
\end{equation}

\begin{proposition}[Logarithmic determinant and comparability]\label{lem:logdet}
For $n\ge n_0$ and $|z|\le1$,
\begin{equation}\label{eq:logdet}
\frac{\beta_n(z)}{\sigma_0(n,z)}=\frac{\det\mathcal D(z)}{\det(\mathcal D(z)-a\mathcal A)},
\qquad
\log\frac{\beta_n(z)}{\sigma_0(n,z)}=\sum_{r\ge1}a^{2r}T_r(z).
\end{equation}
Moreover, there is $C>0$ such that $|T_r(z)|\le C^rn^{-2r}$ for all $r\ge1$
and $C^{-1}M_n\le|\beta_n(z)|\le CM_n$. Hence
$\sum_{r\ge4}|a|^{2r}\sup_{|z|\le1}|T_r(z)|=O(n^{-8})$, and
$\sup_{|z|\le\rho}|\beta_n^{(k)}(z)|=O(M_n)$ for fixed $0<\rho<1$ and $k\ge1$.
\end{proposition}

\begin{proof}
In the height coordinate of \S\ref{subsec:backsteps}, removing the first and
last steps of a walk from $-n$ to $n$ leaves a walk $u_0=1,\dots,u_m=n-1$ in
the path graph, of weight $a^2\cdot a^m\prod_{i=0}^md_{u_i}^{-1}$ by
\eqref{eq:fourier-walk-weight}. Since
$\|a\mathcal D^{-1}\mathcal A\|
\le 2|a|/(4(n-1)-|z|)<1$ on $|z|\le1$ for large $n$,
summing over walks gives
$\beta_n=a^2\bigl[\sum_{m\ge0}(a\mathcal D^{-1}\mathcal A)^m\mathcal D^{-1}\bigr]_{1,n-1}
=a^2\bigl[(\mathcal D-a\mathcal A)^{-1}\bigr]_{1,n-1}$. By Cramer's rule this
entry is $a^{n-2}/\det(\mathcal D-a\mathcal A)$, since deleting the last row
and the first column leaves a triangular matrix with diagonal entries $-a$,
and the cofactor sign cancels their minus signs; with \eqref{eq:sigma0} this
gives the first identity in \eqref{eq:logdet}. Hence
$\beta_n/\sigma_0=\det(I-a\mathcal D^{-1}\mathcal A)^{-1}
=\exp\sum_{m\ge1}\frac{a^m}m\operatorname{tr}[(\mathcal D^{-1}\mathcal A)^m]$,
and the traces with odd $m$ vanish because closed walks in a path graph have
even length.

For the bounds, put
$\mathcal B=|\mathcal D|^{-1/2}\mathcal A|\mathcal D|^{-1/2}$. Taking absolute
values in each closed walk gives
$|\operatorname{tr}[(\mathcal D^{-1}\mathcal A)^{2r}]|\le\operatorname{tr}\mathcal B^{2r}
\le\|\mathcal B\|^{2r-2}\operatorname{tr}\mathcal B^2$. On $|z|\le1$ we have
$|d_k(z)|\ge2k(n-k)\ge2(n-1)$, so $\|\mathcal B\|\le(n-1)^{-1}$, and
$\operatorname{tr}\mathcal B^2=2\sum_k|q_k|=O(n^{-2})$ by \eqref{eq:Omega}.
This proves the bound on $T_r$. For large $n$ we have
$C|a|^2n^{-2}\le\frac12$, so summing the geometric series gives
$\sum_{r\ge4}|a|^{2r}\sup_{|z|\le1}|T_r(z)|\le2C^4|a|^8n^{-8}$. The bound on
$T_r$ also gives $\sum_{r\ge1}a^{2r}T_r(z)=O(n^{-2})$, uniformly on $|z|\le1$.
Hence
$\beta_n/\sigma_0=1+O(n^{-2})$, and for large $n$ the series is the principal
logarithm of this quotient, which gives the second identity in
\eqref{eq:logdet}. Since
$\sigma_0(n,z)/\sigma_0(n,0)=\prod_k(1+z/D_k)^{-1}$ and $U_1\to0$, also
$|\sigma_0(n,z)|\asymp M_n$; this gives the bounds on $|\beta_n|$. Finally,
$\beta_n$ is analytic on a neighbourhood of $|z|\le1$ (Section~\ref{sec:LS}),
so Cauchy's estimate on circles of radius $1-\rho$ gives
$\sup_{|z|\le\rho}|\beta_n^{(k)}(z)|\le k!(1-\rho)^{-k}CM_n$.
\end{proof}

\subsection{The off-diagonal entry at the centre}

We expand the logarithms of the monotone-walk ratio
$\sigma_0(n,z_n^0)/\sigma_0(n,0)$ and the backstep factor
$\beta_n(z_n^0)/\sigma_0(n,z_n^0)$ in powers of $a^2$, with coefficients
\begin{equation}\label{eq:m-nu-def}
\begin{aligned}
m_1&:=-\xi_1U_1,\qquad m_2:=-\xi_2U_1+\tfrac12\xi_1^2U_2,\qquad
m_3:=-\xi_3U_1+\xi_1\xi_2U_2-\tfrac13\xi_1^3U_3,\\
\nu_1&:=Q_n(0),\qquad \nu_2:=\xi_1Q_n'(0)+T_2(0),\qquad
\nu_3:=\xi_2Q_n'(0)+\tfrac12\xi_1^2Q_n''(0)+\xi_1T_2'(0)+T_3(0),
\end{aligned}
\end{equation}
where $\xi_1,\xi_2,\xi_3$ are as in \eqref{eq:xi}.

\begin{proposition}[Off-diagonal entry at the centre]\label{prop:beta}
For large $n$,
\begin{align}
\log\frac{\sigma_0(n,z_n^0)}{\sigma_0(n,0)}&=m_1a^2+m_2a^4+m_3a^6+O(n^{-8}),
\label{eq:monotone-trunc}\\
\log\frac{\beta_n(z_n^0)}{\sigma_0(n,z_n^0)}&=\nu_1a^2+\nu_2a^4+\nu_3a^6+O(n^{-8}),
\label{eq:backstep-trunc}
\end{align}
and $m_r,\nu_r=O(n^{-2r})$ for $r=1,2,3$.
\end{proposition}

\begin{proof}
Put $w:=\xi_1a^2+\xi_2a^4+\xi_3a^6=O(n^{-2})$, so that $z_n^0=w+O(n^{-8})$ by
\eqref{eq:center-trunc}. Since $D_k\ge2n\min(k,n-k)$, we have
$U_1=O(n^{-1}\log n)$ and $U_j=O(n^{-j})$ for $j\ge2$. By \eqref{eq:T-local},
Proposition~\ref{lem:logdet} and Cauchy's estimate, $Q_n^{(j)}=O(n^{-2})$,
$T_2^{(j)}=O(n^{-4})$ and $T_3^{(j)}=O(n^{-6})$ on $|z|\le\frac12$.

\emph{Monotone-walk ratio.} By \eqref{eq:sigma0}, the left side of
\eqref{eq:monotone-trunc} is $-\sum_k\log(1+z_n^0/D_k)$. For $|z|\le\frac12$,
the derivative of $-\sum_k\log(1+z/D_k)$ is $O(U_1)$. Replacing $z_n^0$ by $w$
therefore costs $O(n^{-9}\log n)$, and
$-\sum_k\log(1+w/D_k)=-wU_1+\frac12w^2U_2-\frac13w^3U_3+O(n^{-12})$. By
\eqref{eq:m-nu-def}, the terms of order $a^2$, $a^4$, $a^6$ are
$m_1a^2$, $m_2a^4$, $m_3a^6$, and the rest are $O(n^{-8})$.

\emph{Backstep factor.} By Proposition~\ref{lem:logdet} at $z=z_n^0$ and
\eqref{eq:T-local}, the left
side of \eqref{eq:backstep-trunc} is
$a^2Q_n(z_n^0)+a^4T_2(z_n^0)+a^6T_3(z_n^0)+O(n^{-8})$, and replacing $z_n^0$ by
$w$ costs $O(n^{-10})$. When $Q_n$, $T_2$ and $T_3$ are expanded about $0$ to orders
$2$, $1$ and $0$, the terms of order $a^2$, $a^4$, $a^6$ are
$\nu_1a^2$, $\nu_2a^4$, $\nu_3a^6$ by \eqref{eq:m-nu-def}, and the rest are
$O(n^{-8})$.

The bounds on $m_r$ and $\nu_r$ follow from those on $\xi_r$, $U_j$, $Q_n$,
$T_2$ and $T_3$.
\end{proof}

\section{Proof of the main theorem and isospectral extension}\label{sec:proof}

We first express the splitting through the off-diagonal entry and inverse
slope at the centre $z_n^0$ of Lemma~\ref{lem:center}
(\S\ref{subsec:root}), then combine their expansions from
Sections~\ref{sec:diagonal} and~\ref{sec:offdiag} to prove
Theorem~\ref{thm:main} (\S\ref{subsec:completion}).
The isospectral extension to $ce^{-2ix}+de^{2ix}$ with $a^2=cd\ne0$
is proved in \S\ref{subsec:asymmetric}.

\subsection{Root separation}\label{subsec:root}

By \eqref{eq:det}, $\Delta_n=(F_n-\beta_n)(F_n+\beta_n)$. We locate one root
of each factor and expand both about $z_n^0$, where $F_n$ vanishes.

\begin{proposition}[Root separation]\label{prop:root}
For large $n$, each equation
$F_n(z)=\pm\beta_n(z)$ has a unique solution $z_n^\pm$ in $|z|\le\frac12$.
These solutions satisfy
\[
z_n^\pm=z_n^0+O(M_n),\qquad \tfrac12(z_n^++z_n^-)=z_n^0+O(M_n^2),
\]
and
\begin{equation}\label{eq:root-reduction}
z_n^+-z_n^-=\frac{2\beta_n(z_n^0)}{F_n'(z_n^0)}\bigl[1+O(M_n^2)\bigr].
\end{equation}
\end{proposition}

\begin{proof}
Take $\rho=\frac12$ in Propositions~\ref{prop:diagonal-tail-uniform}
and~\ref{lem:logdet}: on $|z|\le\frac12$,
$\alpha_n^{(k)}=O(n^{-2})$ for $k\le3$ and $\beta_n^{(k)}=O(M_n)$ for $k\le2$.
Hence, as in the proof of Lemma~\ref{lem:center}, for large $n$ each map
$\alpha_n\pm\beta_n$ is a $\frac12$-contraction of $\{|z|\le\frac12\}$ into
itself, and its fixed point is the unique solution $z_n^\pm$ there. Put
$\delta_\pm:=z_n^\pm-z_n^0$. Since $z_n^0=\alpha_n(z_n^0)$ and $\alpha_n$ is a
$\frac12$-contraction, $|\delta_\pm|\le\frac12|\delta_\pm|+|\beta_n(z_n^\pm)|$,
so $\delta_\pm=O(M_n)$.

Expand $F_n$ about $z_n^0$ to second order and $\beta_n$ to first order. Since
$F_n(z_n^0)=0$, and the remainders $O(n^{-2}|\delta_\pm|^3)$ and
$O(M_n|\delta_\pm|^2)$ are $O(M_n^3)$, the equations
$F_n(z_n^\pm)=\pm\beta_n(z_n^\pm)$ become
\[
F_n'(z_n^0)\,\delta_\pm+\tfrac12F_n''(z_n^0)\,\delta_\pm^2
=\pm\bigl(\beta_n(z_n^0)+\beta_n'(z_n^0)\,\delta_\pm\bigr)+O(M_n^3),
\]
where $|F_n'(z_n^0)|\ge\frac12$, $F_n''=-\alpha_n''=O(n^{-2})$ and
$\beta_n'=O(M_n)$. Adding the two equations gives
$F_n'(z_n^0)(\delta_++\delta_-)=O(M_n^2)$, so $\delta_++\delta_-=O(M_n^2)$.
Subtracting them, and using this bound in the terms
$\beta_n'(z_n^0)(\delta_++\delta_-)$ and
$\frac12F_n''(z_n^0)(\delta_+-\delta_-)(\delta_++\delta_-)$, gives
$F_n'(z_n^0)(z_n^+-z_n^-)=2\beta_n(z_n^0)+O(M_n^3)$. Since
$|\beta_n(z_n^0)|\ge C^{-1}M_n$ by Proposition~\ref{lem:logdet}, this is
\eqref{eq:root-reduction}.
\end{proof}

\subsection{Proof of Theorem~\ref{thm:main}}\label{subsec:completion}

We first use root separation and the eigenvalue count to show that
$n^2+z_n^\pm$ are the only eigenvalues in $|\lambda-n^2|<1$ and are
algebraically simple, and then combine the logarithmic expansions to obtain
\eqref{eq:main-refined}.

\begin{proof}[Proof of Theorem~\ref{thm:main}]
Since $\beta_n(z_n^0)\ne0$, \eqref{eq:root-reduction} gives
$z_n^+\ne z_n^-$, so by \eqref{eq:eig-equiv} with $R=\frac12$, $n^2+z_n^+$ and $n^2+z_n^-$
are two distinct eigenvalues in $|\lambda-n^2|\le\frac12$. By
\cite[Lemma~1]{AnahtarciDjakov2012}, the disc $|\lambda-n^2|<1$ contains
exactly two eigenvalues, counted with algebraic multiplicity. Hence
$n^2+z_n^\pm$ are the only eigenvalues in that disc, both are algebraically
simple, and $\gamma_n=\pm(z_n^+-z_n^-)$.

Dividing \eqref{eq:root-reduction} by $2\sigma_0(n,0)\ne0$, and using that
$M_n^2=O(n^{-8})$ because $M_n$ decays factorially, we obtain
\[
\frac{z_n^+-z_n^-}{2\sigma_0(n,0)}
=\frac{1}{F_n'(z_n^0)}\cdot
\frac{\sigma_0(n,z_n^0)}{\sigma_0(n,0)}\cdot
\frac{\beta_n(z_n^0)}{\sigma_0(n,z_n^0)}\cdot\bigl[1+O(n^{-8})\bigr].
\]
By \eqref{eq:inverse-alpha}, \eqref{eq:monotone-trunc} and
\eqref{eq:backstep-trunc}, the three factors have logarithms of size
$O(n^{-2})$. Adding them gives, with the principal logarithm on the left,
\begin{equation}\label{eq:log-assembly}
\log\frac{z_n^+-z_n^-}{2\sigma_0(n,0)}=c_1a^2+c_2a^4+c_3a^6+O(n^{-8}),
\qquad c_r:=s_r+m_r+\nu_r .
\end{equation}
By Lemma~\ref{lem:coefficients}, $c_1(n)=-(n+2)/[4(n^2-1)^2]$,
$c_2(n)=-15/(128n^7)+O(n^{-8})$ and $c_3(n)=O(n^{-8})$. Since
$2\sigma_0(n,0)=8(a/4)^n/[(n-1)!]^2$ by \eqref{eq:sigma0} and
$\gamma_n=\pm(z_n^+-z_n^-)$, this proves \eqref{eq:main-refined}. Its exponent
is $O(n^{-3})$; expanding the exponential to second order, with
\[
c_1(n)=-\frac{1}{4n^3}-\frac{1}{2n^4}-\frac{1}{2n^5}-\frac{1}{n^6}-\frac{3}{4n^7}+O(n^{-8}),
\qquad
\tfrac12c_1(n)^2=\frac{1}{32n^6}+\frac{1}{8n^7}+O(n^{-8}),
\]
gives \eqref{eq:main-expanded}.

For real $a$, $L$ is self-adjoint, and by the min--max principle its ordered
eigenvalues differ from those of $-d^2/dx^2$, under the same boundary
conditions, by at most $\|2a\cos2x\|_\infty=2|a|$. The free value $n^2$
occurs exactly twice, at the positions of the edges of the $n$th gap
\cite{MagnusWinkler1966}; every other eigenvalue lies within $2|a|$ of a free
value $k^2\ne n^2$, hence at distance at least $4(n-1)-2|a|>1$ from $n^2$ for
large $n$. So $n^2+z_n^\pm$ are the edges of the $n$th gap. If
$\lambda_n^-<\lambda_n^+$, then $\gamma_n$ and the prefactor in
\eqref{eq:main-refined} are real, so the exponential factor is real and, being
$1+O(n^{-3})$, positive. Taking absolute values in \eqref{eq:main-refined}
gives the sign $+$ and $(|a|/4)^n$ in place of $(a/4)^n$.
\end{proof}

Lemma~\ref{lem:coefficients} evaluates $c_1$ and $c_2$ exactly, showing that
the harmonic sums cancel, and bounds $c_3$. At each order $a^{2r}$,
$r=1,2,3$, the leading term
$-2/(r^2(16n^2)^r)$ of the slope coefficient $s_r$ cancels the leading term
$2/(r^2(16n^2)^r)$ of $T_r(0)$. The latter comes from the two ends of the
path: at the left end, $q_k(0)\sim[16n^2(k-1)k]^{-1}$ as $n\to\infty$ for each
fixed $k\ge2$, and the right end follows by reflection; for $r=2,3$ the
constant $\pi^2$ cancels in these end sums. Further cancellations at orders $n^{-5}$
and $n^{-6}$ (for instance, the term $1/(64n^5)$ of $T_2(0)$ against
$\xi_1Q_n'(0)$) give $c_2=O(n^{-7})$; at order $a^6$, the terms of order
$n^{-7}$ of $T_3(0)$ and $\xi_1T_2'(0)$ cancel, giving $c_3=O(n^{-8})$.

\subsection{Asymmetric two-term potential}\label{subsec:asymmetric}

The family $ce^{-2ix}+de^{2ix}$ contains the Mathieu potential when
$c=d=a$; the corresponding operator is self-adjoint exactly when
$d=\bar c$.

\begin{corollary}[Asymmetric two-term potential]\label{cor:asymmetric}
For $c,d\in\C$, let $\tilde L y=-y''+(ce^{-2ix}+de^{2ix})y$ on $[0,\pi]$,
with periodic boundary conditions for even $n$ and anti-periodic ones for
odd $n$.

\emph{(a)} If $cd\ne0$, then for large $n$ the operator $\tilde L$ has exactly
two eigenvalues $\lambda_n^-,\lambda_n^+$ in $|\lambda-n^2|<1$, both
algebraically simple, and
\begin{equation}\label{eq:asymmetric}
\gamma_n=\lambda_n^+-\lambda_n^-
=\pm\frac{8(\sqrt{cd}/4)^n}{[(n-1)!]^2}
\exp\!\left[-\frac{cd\,(n+2)}{4(n^2-1)^2}-\frac{15(cd)^2}{128n^7}+O(n^{-8})\right],
\end{equation}
where $\sqrt{cd}$ is either square root of $cd$. The implied constant depends
only on $cd$ and, for each $n$, the sign depends on the labelling of the pair
and on the choice of square root. If $cd>0$, both eigenvalues are real, and if
$\lambda_n^-<\lambda_n^+$, then \eqref{eq:asymmetric} holds with the sign $+$
and the positive square root.

\emph{(b)} If $cd=0$, then for every $n\ge1$, $n^2$ is an eigenvalue of
$\tilde L$ of algebraic multiplicity two, and $\gamma_n=0$.
\end{corollary}

\begin{proof}
By \cite[Theorem~1]{Veliev2013Isospectral}, rescaled from period $1$ to
period $\pi$ by replacing Veliev's variable with $x/\pi$ and dividing by $\pi^2$, the
potentials $ce^{-2ix}+de^{2ix}$ with a given product $cd$ all have the same
periodic and anti-periodic eigenvalues. Since $2a\cos2x=ae^{-2ix}+ae^{2ix}$,
the operators $\tilde L$ and $L$ have the same eigenvalues for either square
root $a$ of $cd$.

(a) For large $n$, Theorem~\ref{thm:main} gives two distinct eigenvalues of
$\tilde L$ in $|\lambda-n^2|<1$. The count of \cite[Lemma~1]{AnahtarciDjakov2012}
holds for every complex potential in $L^2(0,\pi)$ with zero mean, so this disc contains
exactly two eigenvalues of $\tilde L$, counted with algebraic multiplicity. As
in the proof of Theorem~\ref{thm:main}, both are algebraically simple, and
\eqref{eq:main-refined} with $a=\sqrt{cd}$ is \eqref{eq:asymmetric}, with the
same constant. If $cd>0$, take $a=\sqrt{cd}>0$: then $L$ is self-adjoint, so
the eigenvalues are real, and the last assertion of Theorem~\ref{thm:main}
gives the sign $+$. For $cd\ne0$, \eqref{eq:asymmetric} refines
\cite[(1.13)]{AnahtarciDjakov2012}.

(b) For $0\le t\le1$, let $L_ty=-y''+t(ce^{-2ix}+de^{2ix})y$. If $c=0$, the
coefficient equations of $(L_t-\lambda)y=0$ in the basis $e^{ikx}$,
$k\in n+2\mathbb Z$, read $(k^2-\lambda)y_k=-td\,y_{k-2}$. For $td\ne0$ they
give $|y_{k-2}|>|y_k|$ whenever $y_k\ne0$ and $k^2>|\lambda|+|td|$; since
$y_k\to0$ as $k\to-\infty$, a nonzero solution in $L^2(0,\pi)$ has a lowest
nonzero coefficient $y_{k_0}$,
and its equation gives $\lambda=k_0^2$; for $td=0$ this is clear. If $d=0$,
the highest nonzero coefficient plays the same role. Since $L_t$ has compact
resolvent, its spectrum therefore lies in $\{k^2:k\in n+2\mathbb Z\}$. As
$k^2-n^2\in4\mathbb Z$, the circle $|\lambda-n^2|=1$ lies in the resolvent set
of every $L_t$, and the Riesz projection of $L_t$ for this circle depends
continuously on $t$ in norm. Projections at norm distance less than $1$ have
equal rank, so its rank is constant, equal to $2$ at $t=0$. Hence $n^2$ is an
eigenvalue of $\tilde L=L_1$ of algebraic multiplicity two, and $\gamma_n=0$.
\end{proof}

\section{Conclusion}\label{sec:conclusion}

For fixed $a$, Theorem~\ref{thm:main} refines \eqref{eq:ad}: it gives the gap,
or splitting, as the leading factor times an explicit exponential, which
determines the relative corrections of orders $n^{-4}$ through $n^{-7}$, with
relative error $O(n^{-8})$ instead of $O(n^{-4})$. We keep the framework of
\cite{AnahtarciDjakov2012} and replace two $O(n^{-4})$ estimates in its
argument: the walks from $-n$ to $n$ are summed exactly, relative to the
monotone walk, as the reciprocal of a tridiagonal determinant, and both
eigenvalues are expanded about a common centre, so the diagonal entry need not
be compared at the two eigenvalues. The diagonal entry is expanded through the
closed walks of length six, using a reflection pairing, and the splitting is
assembled from the logarithms of three factors at the centre.

\appendix
\numberwithin{equation}{section}

\section{Exact finite-sum identities}\label{app:sums}

Lemma~\ref{lem:coefficients} uses the identities below, together with
$U_1=H_{n-1}/(2n)$, to compute $c_1$ and $c_2$ in
\eqref{eq:log-assembly}. We give closed forms for $Q_n(0)$, $U_2$,
$Q_n'(0)$ and the two sums defining $T_2(0)$ in \eqref{eq:T-local}.

Write $H_m^{(j)}:=\sum_{r=1}^m r^{-j}$ for integers $m\ge0$ and
$j\ge1$, with $H_m=H_m^{(1)}$ and empty sums equal to zero.
As in Section~\ref{sec:offdiag}, let $D_k=4k(n-k)$.
In Appendices~\ref{app:sums} and~\ref{app:coupling},
$q_k$ without an argument denotes
$q_k(0)=1/(D_{k-1}D_k)$.

\begin{lemma}[Exact finite-sum identities]\label{lem:exact-sums}
For every integer $n\ge2$,
\begin{align}
Q_n(0)
&=\frac{n(n-3)+2(n-1)H_{n-1}}{8n(n-1)^2(n+1)},
\label{eq:Q-exact}\\[2pt]
U_2
&=\frac{H_{n-1}}{4n^{3}}+\frac{H_{n-1}^{(2)}}{8n^{2}},
\label{eq:sumDsq}\\[2pt]
Q_n'(0)
&=-\frac{(n-3)(n+5)}{32(n-1)^{3}(n+1)^{2}}
 -\frac{(3n^{2}-1)\,H_{n-1}}{8n^{3}(n-1)^{2}(n+1)^{2}}
 -\frac{H_{n-1}^{(2)}}{16n^{2}(n-1)(n+1)},
\label{eq:Qprime-exact}\\[2pt]
\sum_{k=2}^{n-1}q_k^{2}
&=-\frac{3n^{3}-7n^{2}-7n+35}{128(n-1)^{4}(n+1)^{3}}
 +\frac{(5n^{2}-1)\,H_{n-1}}{32n^{3}(n-1)^{3}(n+1)^{3}}
 +\frac{(n^{2}+1)\,H_{n-1}^{(2)}}{64n^{2}(n-1)^{2}(n+1)^{2}},
\label{eq:qsq-exact}
\end{align}
and, for every integer $n\ge3$,
\begin{equation}
\begin{aligned}
\sum_{k=3}^{n-1}q_{k-1}q_k
&=\frac{7n^{4}-21n^{3}+7n^{2}-87n+166}
 {512(n-2)^{2}(n-1)^{3}(n+1)^{2}(n+2)}\\
&\qquad
 +\frac{(5n^{2}-2)\,H_{n-1}}{32n^{3}(n-2)(n-1)^{2}(n+1)^{2}(n+2)}
 -\frac{H_{n-1}^{(2)}}{128n^{2}(n-1)(n+1)}.
\end{aligned}
\label{eq:qadj-exact}
\end{equation}
\end{lemma}

\begin{proof}
For fixed $n$ in the stated range, we decompose each summand into partial
fractions in the summation index. The summands are rational functions that
vanish at infinity, and their poles, as written, are distinct, so each summand
is the sum of its principal parts. At a double pole $p$, if the summand is
$g(x)(x-p)^{-2}$, the coefficients of $(x-p)^{-2}$ and $(x-p)^{-1}$ are $g(p)$
and $g'(p)=g(p)\,(\log g)'(p)$. Summing term by term gives differences of
harmonic numbers, which we express in terms of $H_{n-1}^{(j)}$ and rational
functions of $n$ using $H_{n-1-i}^{(j)}=H_{n-1}^{(j)}-\sum_{r=n-i}^{n-1}r^{-j}$.

\emph{Proof of \eqref{eq:Q-exact}.} Since
$16q_k=[(k-1)(n+1-k)k(n-k)]^{-1}$, partial fractions give
\[
q_k
=\frac{1}{16n}\left[
\frac{1}{n-1}\left(\frac{1}{k-1}+\frac{1}{n-k}\right)
-\frac{1}{n+1}\left(\frac{1}{k}+\frac{1}{n+1-k}\right)
\right].
\]
Summing over $2\le k\le n-1$ gives
$Q_n(0)=\frac{1}{8n}\bigl[\frac{H_{n-2}}{n-1}-\frac{H_{n-1}-1}{n+1}\bigr]$,
which reduces to \eqref{eq:Q-exact}.

\emph{Proof of \eqref{eq:sumDsq}.} Squaring
$\frac{1}{k(n-k)}=\frac1n\bigl(\frac1k+\frac1{n-k}\bigr)$ gives
$\frac{16}{D_k^2}=\frac{1}{n^2}\bigl(\frac{1}{k^2}+\frac{1}{(n-k)^2}\bigr)
+\frac{2}{n^3}\bigl(\frac1k+\frac1{n-k}\bigr)$, and summing over
$1\le k\le n-1$ gives \eqref{eq:sumDsq}.

\emph{Proof of \eqref{eq:Qprime-exact}.} By \eqref{eq:dk} and \eqref{eq:qk},
$q_k'(0)=-q_k(D_{k-1}^{-1}+D_k^{-1})$, and the reflection $k\mapsto n+1-k$
of $\{2,\ldots,n-1\}$ exchanges $D_{k-1}$ and $D_k$. Hence
\[
Q_n'(0)=-2\sum_{k=2}^{n-1}\frac{1}{D_{k-1}D_k^{2}}
=-\frac{1}{32}\sum_{k=2}^{n-1}\frac{1}{(k-1)(n+1-k)k^{2}(n-k)^{2}},
\]
and the summand decomposes as
\begin{align*}
&\frac{1}{n(n-1)^2}\,\frac{1}{k-1}
 +\frac{1}{n(n+1)^2}\,\frac{1}{n+1-k}
 -\frac{1}{n^2(n+1)}\,\frac{1}{k^2}
 -\frac{n^2+4n+2}{n^3(n+1)^2}\,\frac{1}{k}\\[2pt]
&\quad
 +\frac{1}{n^2(n-1)}\,\frac{1}{(n-k)^2}
 -\frac{n^2-4n+2}{n^3(n-1)^2}\,\frac{1}{n-k}.
\end{align*}
Over $2\le k\le n-1$, the terms in $1/(k-1)$ and $1/(n-k)$ each sum to
$H_{n-2}$, those in $1/k$ and $1/(n+1-k)$ to $H_{n-1}-1$, and those in $1/k^2$
and $1/(n-k)^2$ to $H_{n-1}^{(2)}-1$ and $H_{n-2}^{(2)}$; reducing gives
\eqref{eq:Qprime-exact}.

\emph{Proof of \eqref{eq:qsq-exact}.} With $m=k-1$, $1\le m\le n-2$, we have
\[
256q_k^2=[m(m+1)(n-1-m)(n-m)]^{-2}.
\]
This is invariant under
$m\mapsto n-1-m$, which exchanges the double poles $0,-1$ with $n-1,n$, so
\[
256q_k^2=\frac{\kappa_2}{m^2}+\frac{\kappa_1}{m}+\frac{\mu_2}{(m+1)^2}+\frac{\mu_1}{m+1}
+\bigl(\text{the same with } m\mapsto n-1-m\bigr),
\]
where $\kappa_2=\frac{1}{n^2(n-1)^2}$, $\kappa_1=-\frac{2(n^2-3n+1)}{n^3(n-1)^3}$,
$\mu_2=\frac{1}{n^2(n+1)^2}$ and $\mu_1=\frac{2(n^2+3n+1)}{n^3(n+1)^3}$. Both
halves have the same sum, so
\[
\sum_{k=2}^{n-1}q_k^2
=\frac{1}{128}\Bigl[\kappa_2H_{n-2}^{(2)}+\kappa_1H_{n-2}
+\mu_2\bigl(H_{n-1}^{(2)}-1\bigr)+\mu_1\bigl(H_{n-1}-1\bigr)\Bigr],
\]
which reduces to \eqref{eq:qsq-exact}.

\emph{Proof of \eqref{eq:qadj-exact}.} With $m=k-2$, $1\le m\le n-3$ (an empty
range when $n=3$), we have
$256q_{k-1}q_k=[m(m+1)^2(m+2)(n-2-m)(n-1-m)^2(n-m)]^{-1}$. This is invariant
under $m\mapsto n-2-m$, which exchanges the poles $0,-1,-2$ with $n-2,n-1,n$,
so
\[
256q_{k-1}q_k=\frac{\kappa_1}{m}+\frac{\mu_2}{(m+1)^2}+\frac{\mu_1}{m+1}+\frac{\theta_1}{m+2}
+\bigl(\text{the same with } m\mapsto n-2-m\bigr),
\]
where now $\kappa_1=\frac{1}{2n(n-2)(n-1)^2}$ and
$\theta_1=-\frac{1}{2n(n+2)(n+1)^2}$, while $\mu_2=-\frac{1}{n^2(n-1)(n+1)}$
and $\mu_1=-\frac{2(2n^2-1)}{n^3(n-1)^2(n+1)^2}$. Both halves have the same sum, so
\[
\sum_{k=3}^{n-1}q_{k-1}q_k
=\frac{1}{128}\Bigl[\kappa_1H_{n-3}+\mu_2\bigl(H_{n-2}^{(2)}-1\bigr)
+\mu_1\bigl(H_{n-2}-1\bigr)+\theta_1\bigl(H_{n-1}-\tfrac32\bigr)\Bigr],
\]
which reduces to \eqref{eq:qadj-exact}.
\end{proof}

\section{Exact coefficients and fixed-index checks}\label{app:coupling}

Lemma~\ref{lem:coefficients} completes the proof of Theorem~\ref{thm:main}
by evaluating the coefficients of \eqref{eq:log-assembly};
Proposition~\ref{prop:b2} and Remark~\ref{rem:dlmf} are checks, not used in
that proof.

\begin{lemma}[Exact coefficients]\label{lem:coefficients}
The coefficients $c_1$ and $c_2$ of \eqref{eq:log-assembly} are the rational
functions
\begin{equation}\label{eq:b2-closed}
\begin{aligned}
c_1(n)&=-\frac{n+2}{4(n^2-1)^2}\qquad(n\ge2),\\
c_2(n)&=-\frac{15n^5+76n^4-12n^3-176n^2-111n-188}{128(n^2-4)^2(n^2-1)^4}
\qquad(n\ge3),
\end{aligned}
\end{equation}
and $c_3(n)=O(n^{-8})$. In particular,
$c_1(n)=-\frac{1}{4n^3}-\frac{1}{2n^4}-\frac{1}{2n^5}-\frac1{n^6}-\frac{3}{4n^7}+O(n^{-8})$
and $c_2(n)=-\frac{15}{128n^7}+O(n^{-8})$.
\end{lemma}

\begin{proof}
By \eqref{eq:A1}--\eqref{eq:A5}, the quantities $\xi_r$, $\omega_r$ and
$s_r$ in \eqref{eq:xi} and \eqref{eq:s-def} are rational functions of $n$.
Each of $U_j$, $Q_n^{(j)}(0)$, $T_2^{(j)}(0)$ and $T_3(0)$ is a finite linear
combination of sums over $k$ of products of negative powers of
$D_{k-3},\dots,D_k$. As a rational function of $k$, such a product has poles
only at fixed integers and at $n$ plus fixed integers, and its partial-fraction
coefficients are rational in $n$; so summing over $k$ gives shifted harmonic
numbers, which reduce to $H_{n-1}^{(\ell)}$, $\ell\ge1$, plus rational
functions of $n$. Since
\eqref{eq:m-nu-def} is linear in these sums, each $c_r$ is a linear combination
of $1$ and the $H_{n-1}^{(\ell)}$ with coefficients rational in $n$.

\emph{The coefficient $c_1$.} By \eqref{eq:Q-exact}, the harmonic part of
$Q_n(0)$ is $H_{n-1}/[4n(n^2-1)]=\xi_1U_1$, which cancels $m_1$. Hence, for
every $n\ge2$,
\[
\begin{aligned}
c_1&=s_1+m_1+\nu_1=f_n'(0)+\bigl[Q_n(0)-\xi_1U_1\bigr]\\
&=-\frac{n^2+1}{8(n^2-1)^2}+\frac{n-3}{8(n-1)^2(n+1)}
=-\frac{n+2}{4(n^2-1)^2}.
\end{aligned}
\]

\emph{The coefficient $c_2$.} By Lemma~\ref{lem:exact-sums}, only $H_{n-1}$
and $H_{n-1}^{(2)}$ occur. Over the common denominators $128n^2(n^2-1)^2$ and
$64n^3(n^2-1)^3(n^2-4)$, the coefficients of $H_{n-1}^{(2)}$ and $H_{n-1}$
have numerators
\begin{align*}
&\underbrace{2}_{\frac12\xi_1^2U_2}
-\underbrace{4}_{\xi_1Q_n'(0)}
+\underbrace{(n^{2}+1)}_{\frac12\sum q_k^{2}}
+\underbrace{(1-n^{2})}_{\sum q_{k-1}q_k}=0,\\[2pt]
&\underbrace{(-3n^4-17n^2+8)}_{m_2}
-\underbrace{(12n^4-52n^2+16)}_{\xi_1Q_n'(0)}
+\underbrace{(5n^4-21n^2+4)}_{\frac12\sum q_k^2}
+\underbrace{(10n^4-14n^2+4)}_{\sum q_{k-1}q_k}=0.
\end{align*}
Collecting the rational parts gives the stated $c_2$, which
Remark~\ref{rem:dlmf} confirms at $n=3$. The expansions of $c_1$ and $c_2$
follow from \eqref{eq:b2-closed}.

\emph{The coefficient $c_3$.} Since \eqref{eq:A1}--\eqref{eq:A5} and \eqref{eq:xi} are invariant
under $n\mapsto-n$, $s_3$ is even in $n$; its $n^{-6}$ terms come from
$h_n'(0)$, $\omega_1\omega_2$ and $\omega_1^3/3$, namely $-23/36864$,
$5/4096$ and $-1/1536$, while $\omega_3-h_n'(0)=O(n^{-8})$; hence
$s_3=-1/(18432n^6)+O(n^{-8})$. For $2\le k\le n/2$,
\[
q_k(0)=\frac{1+(2k-1)/n+O(k^2/n^2)}{16n^2(k-1)k},\qquad
q_k'(0)=-\frac{q_k(0)}{4n}\Bigl(\frac1{k-1}+\frac1k\Bigr)\bigl(1+O(k/n)\bigr),
\]
and $q_{n+1-k}(z)=q_k(z)$. With $u_m=1/(m(m+1))$ and $w_m=1/m+1/(m+1)$,
\eqref{eq:T-local} therefore gives $T_3(0)=2(16n^2)^{-3}(S+S'/n)+O(n^{-8})$
and $\xi_1T_2'(0)=-S''/(1024n^7)+O(n^{-8})$ (split each sum at $n/2$ and use
this symmetry; the $O(k^2/n^2)$ and $O(k/n)$ corrections contribute
$O(n^{-8})$ because $\sum_{m\ge1}m^2u_m^3$ and $\sum_{m\ge1}mu_m^2w_m$
converge, and the tails of the sums beyond $n/2$ are smaller), where
\[
\begin{aligned}
S&=\sum_{m\ge1}\Bigl[\tfrac13u_m^3+u_mu_{m+1}(u_m+u_{m+1})+u_mu_{m+1}u_{m+2}\Bigr]\\
&=\frac{10-\pi^2}3+\Bigl(\frac{\pi^2}2-\frac{39}8\Bigr)+\Bigl(\frac{119}{72}-\frac{\pi^2}6\Bigr)
=\frac19
\end{aligned}
\]
and
\begin{align*}
S'&=\sum_{m\ge1}\bigl[(2m+1)u_m^3+(6m+5)u_m^2u_{m+1}+(6m+7)u_mu_{m+1}^2
+(6m+9)u_mu_{m+1}u_{m+2}\bigr]\\
&=\bigl(2\zeta(3)-2\bigr)+\Bigl(\frac{5\pi^2}4-\zeta(3)-\frac{85}8\Bigr)
+\Bigl(\frac{55}4-\frac{5\pi^2}4-\zeta(3)\Bigr)+\frac18=\frac54,\\
S''&=\sum_{m\ge1}\bigl[u_m^2w_m+u_mu_{m+1}(w_m+w_{m+1})\bigr]
=\bigl(2\zeta(3)-2\bigr)+\Bigl(\frac{21}8-2\zeta(3)\Bigr)=\frac58 .
\end{align*}
Hence $T_3(0)=1/(18432n^6)+5/(8192n^7)+O(n^{-8})$ and
$\xi_1T_2'(0)=-5/(8192n^7)+O(n^{-8})$. The remaining terms $\xi_2Q_n'(0)$,
$\frac12\xi_1^2Q_n''(0)$ and $m_3$ are $O(n^{-8})$, because
$\xi_r=O(n^{2-4r})$ by \eqref{eq:xi}, $U_j=O(n^{-j}\log n)$ and each $z$-derivative of
$q_k(z)$ gains a factor $O(n^{-1})$. Thus $c_3=O(n^{-8})$.
\end{proof}

The coefficients $b_1(n):=c_1(n)$ for $n\ge2$ and
\begin{equation}\label{eq:b2-fixed}
b_2(n):=c_2(n)+\tfrac12c_1(n)^2
=\frac{(n+4)\bigl(4n^{5}-15n^{4}-32n^{3}+12n^{2}+64n+111\bigr)}
{128\,(n-2)^{2}(n-1)^{4}(n+1)^{4}(n+2)^{2}}
\end{equation}
for $n\ge3$ are also the coefficients of $a^2$ and $a^4$ in the
fixed-index expansion of Proposition~\ref{prop:b2}. This allows a comparison
with independent small-coupling data.

\begin{proposition}[Fixed-index expansion]\label{prop:b2}
Fix an integer $n\ge2$. For small $|a|>0$, the equations
$F_n(z)=\pm\beta_n(z)$ have unique solutions $z_n^\pm(a)$ in $|z|\le\frac12$;
put $\mathcal G_n(a):=[z_n^+(a)-z_n^-(a)]/[2\sigma_0(n,0)]$. Then the disc
$|\lambda-n^2|<1$ contains exactly two eigenvalues of $L$, counted with
algebraic multiplicity, namely the simple eigenvalues $n^2+z_n^\pm(a)$;
their difference is thus $\pm2\sigma_0(n,0)\,\mathcal G_n(a)$.
As $a\to0$, $\mathcal G_2(a)=1+b_1(2)a^2+O(a^4)$ and, for $n\ge3$,
\[
\mathcal G_n(a)=1+b_1(n)\,a^{2}+b_2(n)\,a^{4}+O(a^{6}).
\]
The implied constants depend on $n$.
\end{proposition}

\begin{proof}
Fix $n$ and take $R=\frac34$. For small $|a|$, \eqref{eq:H} holds, and on
$|z|\le R$ the walk series of Section~\ref{sec:LS} give $\alpha_n=O(a^2)$
and, for $n\ge3$, $\sum_{p\ge3}A_{2p-1}=O(a^6)$, with the same bounds for their
$z$-derivatives on $|z|\le\frac12$ by Cauchy's estimate. The proof of
Proposition~\ref{lem:logdet} applies for every $n\ge2$ under \eqref{eq:H} and gives
$\log(\beta_n/\sigma_0)=a^2T_1+a^4T_2+O(a^6)$ and $|\beta_n|\asymp|a|^n$.
With these bounds in place of Proposition~\ref{prop:diagonal-tail-uniform} and
the tail estimate of Proposition~\ref{lem:logdet}, the proofs of
Lemma~\ref{lem:center}, Propositions~\ref{prop:alpha-derivative},
\ref{prop:beta} and~\ref{prop:root}, and of \eqref{eq:log-assembly}, with the
terms of order $a^6$ omitted, apply for fixed $n$ as $a\to0$, with $M_n$
replaced by $|a|^n$; each remainder there consists of terms of degree at least
six in $a$ and becomes $O(a^6)$. Separately, the factor $1+O(M_n^2)$ in
\eqref{eq:root-reduction} becomes $1+O(|a|^{2n})$. For $n=2$, where
\eqref{eq:A3} does not apply and $O(|a|^{2n})=O(a^4)$, only the $a^2$ terms
are kept and the remainders are $O(a^4)$. The centre $z_n^0$ is now the fixed point of $\alpha_n$ in
$|z|\le\frac12$, and the solutions $z_n^\pm(a)$ of $F_n=\pm\beta_n$ in
$|z|\le\frac12$ satisfy
$\log\mathcal G_2(a)=c_1(2)a^2+O(a^4)$ and, for $n\ge3$,
$\log\mathcal G_n(a)=c_1a^2+c_2a^4+O(a^6)+O(a^{2n})$. Since
$b_2=c_2+\frac12c_1^2$, this gives the stated expansions.

For $|a|<\frac12$, the circle $|\lambda-n^2|=1$ lies in the resolvent set of
$L$, because $\|2a\cos2x\|_\infty=2|a|$ is less than the distance $1$ from
the circle to the spectrum of $-d^2/dx^2$ under the same boundary conditions.
As in the proof of Corollary~\ref{cor:asymmetric}(b), the Riesz projection
for this circle has constant rank $2$, so the disc contains exactly two
eigenvalues, counted with algebraic multiplicity.
Since $\mathcal G_n(a)\to1$ as $a\to0$, the roots $z_n^\pm(a)$
are distinct for sufficiently small $|a|>0$. By \eqref{eq:eig-equiv},
$n^2+z_n^\pm(a)$ are eigenvalues, and the algebraic count above
shows that both are simple.
\end{proof}

\begin{remark}[External checks]\label{rem:dlmf}
The fixed-index, small-coupling expansions below provide independent checks
of our coefficient functions; they do not establish the uniform large-$n$
remainder in Theorem~\ref{thm:main}.
In the DLMF form $y''+(\lambda-2q\cos2x)y=0$ of Mathieu's equation (the DLMF
writes $a$ for $\lambda$), our operator corresponds to $q=a$. For small $q>0$,
the two eigenvalues identified in Proposition~\ref{prop:b2} are the
characteristic values
$a_n^{\mathrm{DLMF}}(q)>b_n^{\mathrm{DLMF}}(q)$; the superscript distinguishes
them from the coupling $a$ and the coefficients $b_k(n)$. By \cite[(28.6.4), (28.6.5)]{DLMF28.6},
$a_2^{\mathrm{DLMF}}-b_2^{\mathrm{DLMF}}=q^2/2-q^4/18+O(q^6)$;
dividing by $2\sigma_0(2,0)=q^2/2$ gives $b_1(2)=-1/9$ (periodic case). By
\cite[(28.6.6), (28.6.7)]{DLMF28.6},
\[
a_3^{\mathrm{DLMF}}-b_3^{\mathrm{DLMF}}
=\frac{q^3}{32}-\frac{5q^5}{8192}-\frac{609q^7}{52428800}+O(q^9),
\]
and dividing by $2\sigma_0(3,0)=q^3/32$ gives $b_1(3)=-5/256$ and
$b_2(3)=-609/1638400$ (anti-periodic case). These checks test the prefactor
$2\sigma_0(n,0)$, the coupling convention and the order of subtraction in
both parities, and, at $n=3$, the quartic coefficient $b_2$. For $n\ge7$,
the coefficients $\xi_1,\xi_2,\xi_3$
of \eqref{eq:xi} coincide with the common coefficients of $q^2$, $q^4$, $q^6$
in the small-coupling expansions \cite[(28.6.14)]{DLMF28.6} of
$a_n^{\mathrm{DLMF}}$ and $b_n^{\mathrm{DLMF}}$; this tests the coefficients
of the centre in Lemma~\ref{lem:center}, but not its uniform remainder.

\end{remark}

\section*{Declaration of competing interest}

The author declares that no known competing financial interests or
personal relationships could have appeared to influence the work
reported in this paper.

\section*{Funding}

This research did not receive any specific grant from funding agencies in the
public, commercial, or not-for-profit sectors.

\section*{Declaration of generative AI and AI-assisted technologies in the manuscript preparation process}

During the preparation of this work, the author used ChatGPT and Codex
(OpenAI) and Claude (Anthropic) to assist with drafting, manuscript review,
consistency checking, revision and writing verification scripts. The author
subsequently reviewed and edited the material as needed and takes full
responsibility for the content of the publication.

\bibliographystyle{plain}
\bibliography{mathieu_gap_refinement}

\end{document}